\documentclass[a4paper]{amsart}
\usepackage{amsmath,amssymb,amsthm}
\usepackage[frame,cmtip,curve,arrow,matrix,line,graph]{xy}
\usepackage[dvips]{graphicx}

\author{Ryo Ohkawa}
\address{Department of Mathematics, Tokyo Institute of
Technology, 2-12-1 Ookayama, Meguro-ku, Tokyo 152-8551, JAPAN}
\email{ookawa@math.titech.ac.jp}
\title{Flips of moduli of stable 
torsion free sheaves with $c_1=1$ on $\mathbb{P}^2$}

\newtheorem{theo}{Theorem}[section]
\newtheorem{coro}[theo]{Corollary}
\newtheorem{prop}[theo]{Proposition}
\newtheorem{defi}[theo]{Definition}
\newtheorem{lemm}[theo]{Lemma}

\def\Stab{\mathop{\mathrm{Stab}}\nolimits}

\def\Coh{\mathop{\mathrm{Coh}}\nolimits}

\def\End{\mathop{\mathrm{End}}\nolimits}

\def\D{D^b}

\def\U{\mathcal{U}}
\def\T{\mathcal{T}}
\def\C{\mathbb{C}}
\def\R{\mathbb{R}}

\def\e{\varepsilon}

\newcommand{\M}{\mathcal{M}}
\def\min{\mathop{\mathrm{min}}\nolimits}
\def\max{\mathop{\mathrm{max}}\nolimits}
\def\fr{\mathop{\mathrm{fr}}\nolimits}

\def\Hom{\mathop{\mathrm{Hom}}\nolimits}
\def\Ext{\mathop{\mathrm{Ext}}\nolimits}

\newcommand{\mo}{\mathcal{O}}
\newcommand{\E}{\mathcal{E}}
\newcommand{\F}{\mathcal{F}}
\newcommand{\PP}{\mathbb{P}}
\newcommand{\Q}{\mathbb{Q}}

\newcommand{\A}{\mathcal{A}}
\newcommand{\Z}{\mathbb{Z}}
\newcommand{\Pic}{\operatorname{Pic}}
\newcommand{\GL}{\operatorname{GL}}
\newcommand{\grp}{{\widetilde{\GL}^+}(2,\R)}

\def\Im{\mathop{\mathrm{Im}}\nolimits}
\def\Re{\mathop{\mathrm{Re}}\nolimits}
\def\im{\mathop{\mathrm{im}}\nolimits}

\def\ker{\mathop{\mathrm{ker}}\nolimits}
\def\coker{\mathop{\mathrm{coker}}\nolimits}
\def\dim{\mathop{\mathrm{dim}}\nolimits}
\def\dimv{\mathop{\mathrm{\underline{dim}}}\nolimits}

\def\rk{\mathop{\mathrm{rk}}\nolimits}

\def\ch{\mathop{\mathrm{ch}}\nolimits}
\def\Mod{\mathop{\mathrm{mod\text{-}}}\nolimits}

\def\exp{\mathop{\mathrm{exp}}\nolimits}

\newcommand{\alphap}{\alpha^{\perp}}

\newcommand{\MP}{M_{\PP^2}}
\newcommand{\MPrn}{M_{\PP^2}(r,1,n)}
\newcommand{\Mv}{\overline{M}_{\PP^2}}
\newcommand{\Mvrn}{\overline{M}_{\PP^2}(r,1,n)}
\newcommand{\mop}{\mo_{\PP^2}}
\newcommand{\Homp}{\Hom_{\PP^2}}
\newcommand{\Homr}{\Hom_{B}}
\newcommand{\Extp}{\Ext_{\PP^2}}
\begin{document}
\maketitle

\begin{abstract}
We study flips of moduli schemes of stable torsion free sheaves $E$ with $c_1(E)=1$ on $\mathbb{P}^2$ as wall-crossing phenomena of moduli schemes of stable modules over certain finite dimensional algebra.
They are described as stratified Grassmann bundles.
\end{abstract}

\begin{center}
\emph{Dedicated to Takao Fujita on the occasion of
his 60th birthday}
\end{center}

\section{Introduction}
\label{sec:1}
\subsection{Background}
We denote by $M_{\PP^2}(r,c_1,n)$ the moduli of semistable torsion free sheaves $E$ on $\PP^2$ with the Chern class $c(E)=(r,c_1,n)\in H^{\ast}(\PP^2,\Z)$.
In this paper we treat the case where $c_1=1$.
In this case semistability and stability for $E$ coincide.
When $n\ge r\ge2$, or $n\ge 2$ and $r=1$, the Picard number of $\MPrn$ is equal to $2$ and we have two birational morphisms from $\MPrn$, which is described below.
 
One is defined by J. Li \cite{L} for general cases. 
We denote by $\MPrn_0$ the open subset of $\MPrn$ consisting of stable vector bundles. 
The Uhlenbeck compactification $\Mvrn$ of $\MPrn_0$ is described set theoretically by $$\Mv(r,1,n)=\sqcup_{i\ge 0}M_{\PP^2}(r,1,n-i)_0\times S^i(\PP^2).$$ 
The map $\pi\colon\MPrn\to\Mvrn\colon E\mapsto \pi(E)$ is defined by 
$$\pi(E):=(E^{\vee\vee},\text{Supp}(E^{\vee\vee}/E))\in M_{\PP^2}(r,1,n-i)_0\times S^i(\PP^2),$$
where $E^{\vee\vee}$ is the double dual of $E$ and $i$ is the length of $E^{\vee\vee}/E$. 
In the case where $r=1$, this morphism is called the Hilbert-Chow morphism $\pi\colon (\PP^2)^{[n]}\to S^n(\PP^2)$ and it is a divisorial contraction when $n\ge 2$.   
In the case where $r\ge 2$, this map is birational since it is an isomorphism on $\MPrn_0$ to its image. 
It is shown that the codimension of the complement of $\MPrn_0$ is equal to 1 when $\MP(r,1,n-1)\neq\emptyset$ (cf. \cite[Proposition~3.23]{M1}). 
Hence this map is a divisorial contraction.
 
The other one is defined by Yoshioka. 
In his paper \cite{Y2} on moduli of torsion free sheaves on rational surfaces, he studied the following morphism $$\psi\colon \MPrn\to\MP(n+1,1,n).$$
For any $E\in\MP(r,1,n)$, $\psi(E)$ is defined by the exact sequence 
\begin{equation}\label{ex1}
0\to \Ext_{\PP^2}^1(E,\mop)^{\vee}
\otimes\mop\to\psi(E)\to E\to 0,
\end{equation}
which is called the universal extension, where $\Ext_{\PP^2}^1(E,\mop)^{\vee}$ is the dual vector space of $\Ext_{\PP^2}^1(E,\mop)$.
Here we have $\Homp(E,\mop)=\Extp^2(E,\mop)=0$ and $(n+1,1,n)\in H^{\ast}(\PP^2,\Z)$ is the Chern class of $$[E]-\chi(E,\mop)[\mop]=[E]+\dim\Extp^1(E,\mop)[\mop]\in K(\PP^2),$$ where $\chi(E,\mop)=\sum_i(-1)^i\dim_\C\Ext^i_{\PP^2}(E,\mop)$.

Furthermore the moduli space $\MPrn$ has 
a stratification 
$$\MPrn=\sqcup_{i=0}^r M^i_{\PP^2}(r,1,n),$$
where 
$M^i_{\PP^2}(r,1,n):=\{E\in\MPrn\mid\dim_{\C}
\Hom_{\PP^2}(\mo_{\PP^2},E)=i\}$
and it is called the \emph{Brill-Noether locus}.
The following theorem is shown in \cite{Y2}. 
\begin{theo}\emph{\bf cf. {\cite[Theorem~5.8]{Y2}}}
\label{yos} 
The following hold.
\begin{itemize}
\item[(1)] There exists an
isomorphism $$M_{\PP^2}^i(r,1,n)
\cong\psi^{-1}\left(M^{n-r+i+1}_{\PP^2}(n+1,1,n)\right).$$
\item[(2)] The restriction of $\psi$ to each strata 
$M_{\PP^2}^i(r,1,n)$ is a $Gr(n-r+i+1,i)$-bundle
over the strata $M^{n-r+i+1}_{\PP^2}(n+1,1,n)$. 
\end{itemize}
\end{theo}
By the above theorem if $n$ is large enough, $\psi$ is a birational 
morphism to the image $\im\psi$ and it is a flipping contraction.
By the theory of the birational geometry
\cite{BCHM} we have the diagram called flip
\begin{equation}\label{dai1}
\xymatrix{
\ar[dr]_{\psi_+}M_+(r,1,n)&&\ar@{-->}[ll]\MP(r,1,n).\ar[dl]
^{\psi}\\
&\im\psi&
}
\end{equation} 

The purpose of this note is to describe 
spaces $M_+(r,1,n)$, $\im\psi$ and the
morphism $\psi_+$ in the above diagram 
using terms of moduli spaces. 
We follow ideas in \cite{O}. We consider $M_{\PP^2}(r,1,n)$ 
as a moduli scheme of semistable modules over 
the finite dimensional algebra 
$\End_{\PP^2}\left(\mo_{\PP^2}(1)\oplus
\Omega_{\PP^2}(3)\oplus\mo_{\PP^2}(2)\right)$ 
and study the wall-crossing
phenomena as the stability changes
using the result of \cite{O} as follows.

\subsection{Main results}
We introduce the exceptional collection 
$$\mathfrak{E}:=\left(\mo_{\PP^2}(1), 
\Omega_{\PP^2}^1(3), \mo_{\PP^2}(2)\right)$$
on $\PP^2$ and put $\E:=\mo_{\PP^2}(1)\oplus\Omega_{\PP^2}^1(3)
\oplus\mo_{\PP^2}(2)$
and $B:=\End_{\PP^2}
(\E).$ We denote abelian categories of coherent sheaves on $\PP^2$ and finitely generated right $B$-modules by $\Coh(\PP^2)$ and $\Mod B$ respectively. 
Then by Bondal's Theorem \cite{Bo}, the functor $\Phi:=\mathbf R \Hom_{\PP^2}(\E,-)$ gives an equivalence 
$$\Phi\colon
\D(\PP^2)\cong \D(B),$$
where $\D(\PP^2)$ and $\D(B)$ are the bonded derived categories of $\Coh(\PP^2)$ and $\Mod B$ respectively.
The equivalence $\Phi$ also induces an isomorphism $\varphi\colon K(\PP^2)\cong K(B)$
between the Grothendieck groups 
of $\Coh(\PP^2)$ and $\Mod B$.

For $\alpha\in K(B)$, we put $$\alpha^{\perp}:=\{\theta\in\Hom_\Z(K(B),\R)\mid\theta(\alpha)=0\}.$$
Any $\theta\in\alphap$ defines a stability condition of $B$-modules $E$ with $[E]=\alpha$.
We denote by $M_B(\alpha,\theta)$ the moduli space of $\theta$-semistable $B$-modules $E$ with $[E]=\alpha$.
In particular we take $$\alpha_{r}=\alpha_{r,n}:=\varphi(n\mo_{\PP^2}(-1)[2]+(2n-r+1)\mo_{\PP^2}[1]+(n-1)\mo_{\PP^2})\in K(B).$$
Here we omit subscription "$n$" although $\alpha_r$ depends on $n$, since we almost always fix $n$ in this paper. 
There exists a wall-and-chamber structure on $\alpha_r^{\perp}$. 

When $n$ is large enough we find two chambers $C_-, C_+$ and a wall $W_0\subset\alpha_{r}^{\perp}$ between them such that the following propositions hold (cf. \S~\ref{sec:3}).
We put $$M_-(\alpha_r):=M_B(\alpha_r,\theta_-),
\ \ M_+(\alpha_r):=M_B(\alpha_r,\theta_+),
\ \ M_0(\alpha_r):=M_B(\alpha_r,\theta_0)$$
for any $\theta_-\in C_-$, $\theta_+\in C_+$
and $\theta_0\in C_0$. 
\begin{prop}\emph{\bf{\cite[Main~Theorem~1.3~(iii)]{O}}}
\label{inpro1}
We have an isomorphism $$M_{\PP^2}(r,1,n)\cong
M_-(\alpha_r)\colon E\mapsto \Phi(E[1]).$$
\end{prop}
We automatically get the following diagram
\begin{equation}\label{dai2}
\xymatrix{M_+(\alpha_r)\ar[dr]^{f_+}
&&\ar[dl]_{f_-}M_-(\alpha_r).\\
&M_0(\alpha_r)&
}
\end{equation}
By analyzing this diagram we see that
diagrams (\ref{dai1}) and (\ref{dai2}) coincide up to
isomorphism. In particular we get the 
following proposition.
\begin{prop}\label{inpro2}
We have isomorphisms
\begin{itemize}
\item[(1)] $M_0(\alpha_r)\cong\im\psi$ and\\
\item[(2)] $M_+(\alpha_r)\cong
M_+(r,1,n)$. 
\end{itemize}
\end{prop}
Proofs of Proposition~\ref{inpro2} are given
in \S~\ref{subsec:3-1} for (1) and in \S~\ref{subsec:3-4}
for (2).
Using the $B$-module $S_0:=\Phi(\mo_{\PP^2}[1])$ we define 
the Brill-Noether locus similar to one in Yoshioka's
theory,
$$M_-^i(\alpha_r)=\{E\in M_-(\alpha_r)\mid
\dim_{\C}\Hom_{B}(S_0,E)=i\},$$
$$M_+^i(\alpha_r)=\{E\in M_+(\alpha_r)\mid
\dim_{\C}\Hom_{B}(E,S_0)=i\}.$$ 
Our situation is similar to \cite{NY} and we have our main theorem.

\begin{theo}\label{inmain}
Assume $n\ge r+2$. Then for each $i$ the following hold. 
\begin{itemize}
\item[(1)] The images $f_+(M^i_+(\alpha_r))$
and $f_-(M^i_-(\alpha_r))$ coincide 
in $M_0(\alpha_r)$.\\

We put $M^i_0(\alpha_r):=f_+(M^i_+(\alpha_r))
=f_-(M^i_-(\alpha_r))$.\\

\item[(2)] We have isomorphisms
$M_0^i(\alpha_r)\cong M_-^0(\alpha_{r-i})
\cong M_+^0(\alpha_{r-i})$.
\item[(3)] We have isomorphisms
$M_+^i(\alpha_r)\cong f^{-1}_+(M_0^i(\alpha_r))$.\\
\item[(4)] The restrcition of $f_+$ to each stratum
$M_+^i(\alpha_r)\to M_0^i(\alpha_r)$
is a $Gr(n-r+i-2,i)$-bundle over $M_0^i(\alpha_r)$. 
\end{itemize}
\end{theo}
Note that $M_+(\alpha_r)\neq\emptyset$ if and only if $n\ge r+2$.
Proofs of Main Theorem~\ref{inmain} are given in \S~\ref{subsec:3-3} for (1) and \S~\ref{subsec:3-6} for the others.
We also give a new proof of Theorem~\ref{yos} using terms of $B$-modules via the isomorphism $M_{\PP^2}(r,1,n)\cong M_-(\alpha_r)$ in \S~\ref{subsec:3-6}.
By these descriptions we see that $M_+(r,1, n)$ is smooth and we can compute Hodge polynomials of $M_+(r,1,n)$ from those of $M_{\PP^2}(r,1,n)$. 

The paper is organized as follows.
In \S 2 we introduce a description of Picard group of $M_{\PP^2}(r,1,n)$ in terms of $\theta$-stability of right $B$-modules. 
In \S 3 we study the wall-crossing phenomena of moduli of $\theta$-semistable right $B$-modules. 
This is described as stratified Grassmann bundles and this gives a proof of Main Theorem~\ref{inmain}.
In the Appendix by using Bridgeland stability we give a proof of Proposition~\ref{p2}, which is similar to \cite[Main Theorem~5.1]{O}.\\

\noindent\emph{\bf{Notation.}}
We fix the following notation in the paper:\\
If $A$ is a matrix we denote by ${}^tA$ the transpose
of $A$. If $V$ is $\C$-vector space then we denote 
by $V^{\vee}$ the dual vector space $\Hom_{\C}(V,\C)$
of $V$ and we also denote by $Gr(V,i)$ the Grassmann
manifold of $i$-dimensional subspaces of V.
We consider the polynomial ring $\C[x_0,x_1,x_2]$ and the tensor product $V\otimes \C[x_0,x_1,x_2]$with a vector space $V$. 
For any monomial $m\in\C[x_0,x_1,x_2]$ we put $$V\otimes m:=\{v\otimes m\in V\otimes \C[x_0,x_1,x_2]\mid v\in V\}.$$
We put $\mathbf x:=(x_{0},x_1, x_2)$ and denote by $V\otimes\mathbf x$ the direct sum $$(V\otimes x_{0})\oplus(V\otimes x_{1})\oplus(V\otimes x_{2})$$ of $V$. 
We denote the $i$-th embedding $V\to V\otimes\mathbf x$ and the $i$-th projection $V\otimes\mathbf x\to V$ by $x_i$ and $x_i^{\ast}$ for $i=0,1,2$, respectively. 
For morphisms $f_{ij}\colon U\to V$ between vector spaces $U$ and $V$ for $i,j=0,1,2$, we denote by the matrix
$\begin{pmatrix}
f_{00}&f_{01}&f_{02}\\
f_{10}&f_{11}&f_{12}\\
f_{20}&f_{21}&f_{22}
\end{pmatrix}
$
the morphism $$f:=\sum_{i,j}x_i\circ f_{ij}\circ x^{\ast}_j\colon U\otimes\mathbf x\to V\otimes\mathbf x.$$
We also use similar notation for vector bundles. 

For any path algebra of quiver with relations, we identify modules over the algebra and representations of the corresponding quiver with relations.

\section{Picard group of $M_{\PP^2}(r,1,n)$}
\label{sec:2}
We introduce an explicit description of the Picard group of $M_{\PP^2}(r,1,n)$ in terms of $B$-modules.

\subsection{Finite dimensional algebra $B$}\label{ex}
Finite dimensional algebra 
$B=\End_{\PP^2}(\E)$ is written as a path algebra of the following quiver  
with relations
$(Q,J)$, where $Q$ is defined as
$$Q:=\xymatrix{\stackrel{v_{-1}}{\bullet}&\ar[l]_{\gamma_i}\stackrel{v_0}{\bullet}&\ar[l]_{\delta_j}\stackrel{v_1}{\bullet},} (i,j=0,1,2)$$ 
and $J$ is generated by the following relations 
\begin{equation}\label{rel}
\gamma_i\delta_j+\gamma_j\delta_i=0, (i,j=0,1,2).
\end{equation}
We identify categories $\D(\PP^2)$ and $\D(B)$ and groups $K(\PP^2)$ and $K(B)$ via $\Phi$ and $\varphi$.
For example, we denote $\mo_{\PP^2}(i-1)[2-i]$ and the corresponding simple $B$-module $$\C v_i=\Phi(\mo_{\PP^2}(i-1)[2-i])$$ by the same symbols $S_i$ for $i=-1,0,1$.

We put $e_i:=[S_i]\in K(B)$ ($i=-1,0,1$). 
Then we have $$K(B)=\Z e_{-1}\oplus\Z e_0 \oplus\Z e_1.$$ 
We denote the dual base by $\{e^{\ast}_{-1},e^{\ast}_0, e^{\ast}_1\}$.
For $\alpha_{-1}, \alpha_0,\alpha_1\in\Z$, by 
\begin{equation}
\label{alpha}
\alpha=\begin{pmatrix}
\alpha_{-1}\\
\alpha_0\\
\alpha_1
\end{pmatrix}\in K(B)
\end{equation}
we denote $\alpha=\alpha_{-1}e_{-1}+\alpha_0e_0+\alpha_1e_1\in K(B)$ and for $\theta^{-1}, \theta^0, \theta^{1}\in\R$, by $$\theta=(\theta^{-1},\theta^0,\theta^1)\in\Hom_{\Z}(K(B),\R),$$ we denote $\theta=\theta^{-1}e_{-1}^{\ast}+\theta^{0}e_{0}^{\ast}+\theta^{1}e_{1}^{\ast}\in\Hom_{\Z}(K(B),\R)$.

\subsection{Moduli of semistable $B$-modules}
For any $\alpha\in K(B)$ and $\theta\in\alpha^{\perp}\otimes\R\subset\Hom_{\Z}(K(B),\R)$, we define $\theta$-stability as follows. Here $$\alpha^{\perp}=\{\theta\in \Hom_{\Z}(K(B),\Z)\mid\theta(\alpha)=0\}.$$

\begin{defi}
A right $B$-module $E$ with $[E]=\alpha$ in $K(B)$ is said to be $\theta$-semistable if for any proper submodule $F\subset E$, the inequality $\theta(F)\ge\theta(E)=0$ holds. If the inequality is always strict, then $E$ is said to be $\theta$-stable. 
\end{defi}
 
By $M_B(\alpha,\theta)$ we denote a moduli 
scheme of $\theta$-semistable $B$-module
$E$ with $[E]=\alpha$. 
We define wall and chamber structure on $\alpha^{\perp}\otimes\R$
as follows. Wall is a ray $W=\R_{\ge 0}\theta^W$ in 
$\alpha^{\perp}\otimes\R$
satisfying that there exists a $\theta^W$-semistable $B$-module
$E$ such that $E$ has a proper submodule $F$ with 
$[F]\notin\Q_{>0}\alpha$ in $K(B)$ and $\theta^W(F)=0$. 
A chamber is a connected component 
of $(\alpha^{\perp}\otimes\R)\setminus\cup W$, where $W$ runs over the 
set of all walls in $\alpha^{\perp}\otimes\R$. 
For any chamber $C\subset\alpha^{\perp}\otimes\R$, the moduli space
$M_B(\alpha,\theta)$ does not depend on the choice of $\theta\in C$.

Here we assume that 
$\alpha\in K(B)$ is indivisible and 
$\theta$ is not on any wall in 
$\alpha^{\perp}$, then there exists a universal family $\U$
of $B$-modules on $M_B(\alpha,\theta)$
\begin{equation}\label{univ1}
\U:=\left(\U_{-1}\xrightarrow{\hspace{0.8cm}\gamma^{\ast}_i\hspace{0.8cm}}\U_0\xrightarrow{\hspace{0.8cm}\delta^{\ast}_j\hspace{0.8cm}}\U_1\right), (i,j=0,1,2)
\end{equation}
where $\U_{-1}$, $\U_0$ and $\U_1$ are vector bundles corresponding to vertices $v_{-1},v_0, v_1$ and $\gamma_i^{\ast}\colon \U_{-1}\to \U_0$, $\delta_j^{\ast}\colon \U_0\to \U_1$ are morphisms corresponding to arrows $\gamma_i$, $\delta_j$.

\subsection{Deformations of $B$-modules}
We take $\alpha\in K(B)$ defined by (\ref{alpha}). 
For any $B$-module $E$ with $[E]=\alpha$, by choosing basis of $Ev_{-1}$, $Ev_0$ and $Ev_1$ we have an isomorphism
\begin{equation}\label{mode}
E\cong (\C^{\alpha_{-1}}\stackrel{C_i}{\to}\C^{\alpha_{0}}
\stackrel{D_j}{\to}\C^{\alpha_{1}}),
\end{equation}
where $C_i\in\Hom_{\C}(\C^{\alpha_{-1}},\C^{\alpha_{0}})$
and $D_j\in\Hom_{\C}(\C^{\alpha_{0}},\C^{\alpha_{1}})$
correspond to the action of $\gamma_i$ and $\delta_j$
respectively for $i,j=0,1,2$.
The pull back of the heart $\Mod B$ of the standard t-structure of $\D(B)$ by $\Phi$ is a full subcategory $\mathcal{A}:=\langle\mo(-1)[2], \mo[1],\mo(1)\rangle$ of $\D(\PP^2)$.
The following complex of
coherent sheaves on $\PP^2$  
$$\mo(-1)^{\alpha_{-1}}\xrightarrow{\sum_iC_ix_i}
\mo^{\alpha_{0}}\xrightarrow{\sum_iD_jx_j}
\mo(1)^{\alpha_1}$$
corresponds to $E$ in (\ref{mode}) via the equivalence $\Phi$,
where $x_0, x_1, x_2$ are homogeneous coordinates
of $\PP^2$.
By \cite[Lemma~4.6~(1)]{O}, $\Ext^2_B(E,E)$ is isomorphic to 
the cokernel of the map
\begin{multline}\label{ext}
\left(\bigoplus_i
\Hom_{\C}(\C^{\alpha_{-1}},\C^{\alpha_{0}})x_i\right)
\bigoplus\left(\bigoplus_j
\Hom_{\C}(\C^{\alpha_{0}},\C^{\alpha_{1}})x_j\right)\\
\stackrel{d}{\to}\bigoplus_{i\le j}
\Hom_{\C}(\C^{\alpha_{-1}},\C^{\alpha_{1}})x_ix_j,
\end{multline}
where the map $d$ is defined by
$$d(\sum_i\xi_i x_i, \sum_j\eta_j x_j)=\sum_{i,j}(D_j\xi_i+\eta_iC_j)x_ix_j$$ for $\xi_i\in\Hom_{\C}(\C^{\alpha_{-1}},\C^{\alpha_{0}})\text{ and }\eta_j\in\Hom_{\C}(\C^{\alpha_{0}},\C^{\alpha_{1}}), (i,j=0,1,2).$ 
We study the deformation functor $\mathcal{D}_E\colon$(Artin/$k$) $\to$ (Sets). 
For any Artin local $k$-ring $R$, the set $\mathcal{D}_E(R)$ consists of right $R\otimes B$-modules $E^R$ $$E^R=(R^{\alpha_{-1}}\stackrel{C^R_i}{\to}R^{\alpha_0}\stackrel{D^R_j}{\to}R^{\alpha_1}), D^R_jC^R_i+D^R_iC^R_j=0$$ such that $C^R_i\equiv C_i, D^R_j\equiv D_j$ modulo $m_R$ for each $i,j=0,1,2$, where $C^R_i$ and $D^R_j$ is $R$-linear maps and $m_R$ is the maximal ideal of $R$. We show the following lemma.
\begin{lemm}\label{def}
The deformation functor $\mathcal{D}_E$ has an obstruction theory with values in $\Ext^2_B(E,E)$.
\end{lemm}
\begin{proof}
For any small extension $$0\to\mathfrak{a}\to R'\to R\to0$$ with $m_R\mathfrak{a}=0$ and $E^R=(C^R_i,D^R_j)\in\mathcal{D}_E(R)$, we write $C^R_i=C_i+\xi_i$ and $D^R_j=D_j+\eta_j$ for $\xi_i\in\Hom_{\C}(\C^{\alpha_{-1}},\C^{\alpha_0})\otimes m_R$ and $\eta_j\in\Hom_{\C}(\C^{\alpha_{0}}, \C^{\alpha_1})\otimes m_R$.
By the isomorphism $m_R\cong m_{R'}/\mathfrak{a}$, we have lifts $\xi'_i\in\Hom_{\C}(\C^{\alpha_{-1}}, \C^{\alpha_0})\otimes m_{R'}$ and $\eta'_j\in\Hom_{\C}(\C^{\alpha_{0}},\C^{\alpha_1})\otimes m_{R'}$ of $\xi_i$ and $\eta_j$ respectively.
We put $C^{R'}_j:=C_j+\xi'_i$, $D^{R'}_j:=D_j+\eta'_j$.

Since $D^{R'}_jC^{R'}_i+D^{R'}_iC^{R'}_j
\equiv D^R_jC^R_i+D^R_iC^R_j=0$ modulo
$\mathfrak{a}$, we have an element 
$$\sum_{i\le j}\left(D^{R'}_jC^{R'}_i+D^{R'}_iC^{R'}_j
\right)x_ix_j\in\Hom_{\C}(\C^{\alpha_{-1}},\C^{\alpha_1})
x_ix_j\otimes \mathfrak{a}.$$
By the image of this element to the cokernel
of (\ref{ext}) tensored by $\mathfrak{a}$, 
we define an element 
$\mathfrak{o}(E^R)$ in 
$\Ext_B^2(E,E)\otimes\mathfrak{a}$.
This defines a well-defined map $\mathfrak{o}\colon
\mathcal{D}_E(R)\to \Ext_B^2(E,E)\otimes\mathfrak{a}$
and we easily see that $E^R$ lifts to $\mathcal{D}_E(R')$ if and only if
$\mathfrak{o}(E^R)=0$.
\end{proof}

\subsection{Chambers $C_{\PP^2}$
and Picard group of $M_{\PP^2}(r,1,n)$}
We take \begin{equation}\label{ind}
\alpha_r=
\begin{pmatrix}
n\\2n+r-1\\n-1
\end{pmatrix}
\in K(B)
\end{equation} 
such that $\ch(\alpha_r)=
-(r,1,\frac{1}{2}-n)$ and assume that $M_{\PP^2}(r,1,n)$ 
is not empty set. Then
there exists a chamber $C_{\PP^2}\subset
\alpha_r^{\perp}\otimes\R$ such that $\Phi(\ \cdot\ [1])$ induces an isomorphism
\begin{equation}
\label{isom2}
M_{\PP^2}(r,1,n)\cong M_B(\alpha_r,\theta)
\end{equation}
for any $\theta\in C_{\PP^2}$ 
(see \cite[Main~Theorem~5.1]{O} or \cite{P}).
The chamber $C_{\PP^2}$ is characterized as follows. 
We put $\theta_{\PP^2}:=(-r-1,1,-1+r)$.
Then we have $\theta_{\PP^2}(\mo_x)=0$ 
for any skyscraper sheaf at $x\in\PP^2$.
The closure of $C_{\PP^2}$ contains 
the ray $\R_{\ge 0}\theta_{\PP^2}$ and
for certain $\theta\in C_{\PP^2}$ we have 
$\theta(\mo_x)>0$ and $M_B(\alpha_r,\theta)\neq\emptyset$.

If we put $\theta_0:=(-n+1,0,n)$,
then by \cite[Lemma~6.2]{O} 
we have 
\begin{equation}\label{chaeq}
\R_{>0}\theta_0+\R_{>0}\theta_{\PP^2}
\subset C_{\PP^2}.
\end{equation}
In the case where $r=1$, by \cite[Lemma~6.3~(2)]{O}
we have 
$\R_{>0}\theta_0+\R_{>0}\theta_{\PP^2}=C_{\PP^2}$ for $n\ge 2$ 
In the case where $r\ge 2$, we will describe 
the chamber in \S~\ref{subsec:3-4}.

Since $\alpha_r$ is indivisible
by the definition (\ref{ind}), for 
any $\theta\in C_{\PP^2}$ we have 
a universal family $\U$ on $M_B(\alpha_r, \theta)$ as
in (\ref{univ1}).
We define a homomorphism 
from $\alpha_r^{\perp}$ to $\Pic\left(M_B(\alpha_r,\theta)
\right)$ by 
\begin{equation*}
\rho(\mathbf m)=m_{-1}\det(\U_{-1})+
m_0\det(\U_0)+m_1\det(\U_1),
\end{equation*} 
for $\mathbf m=(m_{-1}, m_{0}, m_{1})\in\alpha_r^{\perp}$. 
By (\ref{isom2}) this
gives a homomorphism $\rho\colon\alpha_r^{\perp}\to
\Pic\left(M_{\PP^2}(r,1,n)\right)$.
Then by \cite{D} we have the following proposition.
\begin{prop}\label{pic}
The above map $\rho\colon\alpha_r^{\perp}\to\Pic 
\left(M_{\PP^2}(r,1,n)\right)$ is an isomorphism.
Furthermore $\rho(-3\theta_{\PP^2})$ is
the canonical bundle of $M_{\PP^2}(r,1,n)$. 
\end{prop}

\section{Proof of Main Theorem~\ref{inmain}}
\label{sec:3}
We put $\alpha_{r,n}:={}^t\left(n,2n-1+r,n-1\right)\in K(B)$.
In the following we omit "$n$" and put $\alpha_r=\alpha_{r,n}$ except in \S~\ref{subsec:3-4}.
Note that $\ch(\alpha_r)=-(r,1,\frac{1}{2}-n)$.
We put $\theta_0:=(-n+1,0,n)\in\alpha^{\perp}_r$ and consider $\theta_+:=\theta_0+\e(2n-1+r,-n,0)$ and $\theta_-:=\theta_0-\e(2n-1+r,-n,0)$ for $\e>0$ small enough such that $\theta_{\pm}$ lie on no wall. 
We put $M_{\pm}(\alpha_r):=M_B(\alpha_r,\theta_{\pm})$ and $M_0(\alpha_r):=M_B(\alpha_r,\theta_0)$.
By (\ref{isom2}) and (\ref{chaeq}) we have an isomorphism 
\begin{equation}\label{iso}
\MPrn\cong M_-(\alpha_r).
\end{equation}
By $C_\pm$ we denote the chamber containing $\theta_\pm$ respectively and put $W_0:=\R_{\ge 0}\theta_0$.
Since $\theta_-\in C_{\PP^2}$, we have $C_-=C_{\PP^2}$.
We automatically get the following diagram: 
\begin{equation}
\label{thad}
\xymatrix{
\ar[dr]^{f_+}M_+(\alpha_r)&&M_-(\alpha_r).\ar[dl]_{f_-}\\
&M_0(\alpha_r)&
}
\end{equation}
In this section we see that this diagram is described by stratified Grassmann bundles and give a proof of Main Theorem~\ref{inmain}.

\subsection{Kronecker modules}
\label{subsec:3-1}
We consider the 3-Kronecker quiver, which has 2 vertices $v_{-1}, v_{1}$ and $3$ arrows $\beta_0,\beta_1,\beta_2$ from $v_{1}$ to $v_{-1}$ $$\begin{minipage}{6cm}\xymatrix{\stackrel{v_{-1}}{\bullet}&&\ar[ll]^{\beta_i}\stackrel{v_1}{\bullet}}\end{minipage}, (i=0,1,2).$$ and consider the path algebra $T$.
Any right $T$-module $G$ has a decomposition $G=Gv_{-1}\oplus Gv_1$ and actions of $\beta_i$ define linear maps $Gv_{-1}\to Gv_1$ for $i=0,1,2$. By abbreviation we define $\theta_0(G)\in\R$ by $$\theta_0(G):=(-n+1)\dim_\C Gv_{-1}+n\dim_\C Gv_1.$$
We denote by $K(T)$ the Grothendieck group of the abelian category of finitely generated right $T$-modules and take $\alpha_T:=n[\C v_{-1}]+(n-1)[\C v_1]\in K(T)$.
\begin{defi}
An right $T$-module $G$ with $[G]=\alpha_T\in K(T)$ is stable if and only if for any non-zero proper submodule $G'$ of $G$ we have an inequality $\theta_0(G')> 0$.
\end{defi}
We denote by $M_T(\alpha_T)$
the moduli space of stable $T$-modules 
$G$ with $[G]=\alpha_T$. 
For any $B$-module $E=\left(\C^n\stackrel{C_i}{\to}
\C^{2n-1+r}\stackrel{D_j}{\to}
\C^{n-1}\right)$, we define $T$-module $E_T$ by 
$$E_T:=\left(\C^n\stackrel{A_i}{\to}\C^{n-1}\right),$$
where $C_i$ and $D_j$ are matrices with 
suitable sizes and we define $A_i$ by 
$A_i:=D_{i+2}C_{i+1}$ for each 
$i\in\Z/3\Z$.
\begin{lemm}\label{lemkro} 
For any $B$-module $E=\left(\C^n\stackrel{C_i}{\to}
\C^{2n-1+r}\stackrel{D_j}{\to}
\C^{n-1}\right)$,
the following hold.
\begin{itemize}
\item[(1)] $E$ is $\theta_0$-semistable if and only if 
$E_T$ is stable.
\item[(2)] $E$ is $\theta_0$-stable if and only if 
$E_T$ is stable and $$\Homr(E,S_0)
=\Homr(S_0,E)=0.$$
\item[(3)] The following are equivalent.
\begin{itemize}
\item[(a$-$)] $E$ is $\theta_-$-stable.  
\item[(b$-$)] $E$ is $\theta_-$-semistable. 
\item[(c$-$)] $E_T$ is stable and $\Homr(E,S_0)
=0$.
\end{itemize}
\item[(4)] The following are equivalent.
\begin{itemize}
\item[(a$+$)] $E$ is $\theta_+$-stable.
\item[(b$+$)] $E$ is $\theta_+$-semistable.
\item[(c$+$)] $E_T$ is stable and $\Homr(S_0,E)
=0$.
\end{itemize}
\end{itemize}
\end{lemm}
\begin{proof}
(1) For every submodule $F\subset E$, we have a submodule $F_T$ of $E_T$.
Conversely for any submodule $G'$ of $E_T$, we define a submodule $F$ of $E$ such that $F_T=G'$ as follows.
We put $Fv_{-1}:=G'v_{-1}$, $Fv_{1}:=G'v_{1}$ and $Fv_{0}:=\sum_i C_i(Fv_{-1})\subset Ev_0$. 
By the relations (\ref{rel}) we have a submodule $F:=Fv_{-1}\oplus Fv_0\oplus Fv_1$ of $E$ and $\theta_0(F)=\theta_0(F_T)=\theta_0(G')$. 
This yields the claim.\\
(2) For any non-zero proper submodule $F\subset E$, $\theta_0(F)=0$ if and only if $\dimv(F)=(0,l,0)$ or $(n,l,n-1)$ for $0<l<2n+r-1$. 
There exists no such $F$ if and only if $\Hom_B(S_0,E)=\Hom_B(E,S_0)=0$.\\
(3) (a$-$)$\implies$(b$-$) It is trivial.
(b$-$)$\implies$(c$-$) We choose $\theta_-=\theta_0-\e(2n-1+r,-n,0)$ for $\e>0$ small enough. 
If $E$ is $\theta_-$-semistable, then for any submodule $F\subset E$ we have $\theta_0(F)\ge 0$, since $\theta_-(F)\ge0$ for arbitrary small $\e>0$. 
This implies that $E$ is $\theta_0$-semistable and hence by (1), $E_T$ is semistable.
Any non-zero $\phi\in\Hom_{B}(E,S_0)$ destabilize $E$. 
Hence we also have $\Hom_{B}(E,S_0)=0$.\\
(c$-$)$\implies$(a$-$) We assume that $E_T$ is stable and $\Hom_{B}(E,S_0)=0$. Hence for every submodule $F\subset E$, we have $\theta_0(F)=\theta_0(F_T)\ge 0$. 
If $\theta_0(F)= 0$ then $\Hom_{B}(E,S_0)=0$ implies that $F\cong S_0^{\oplus l}$ for $0\le l\le 2n+r-1$. 
In this case $\theta_-(F)=\e nl>0$. 
If $\theta_0(F)>0$ then we also have $\theta_-(F)>0$ for $\e$ small enough.
Hence $E$ is $\theta_-$-stable.\\ 
(4) It is similar to the proof of (3).\\
\end{proof}
By the above lemma we have morphisms $\pi_{\pm}^r\colon
M_{\pm}(\alpha_r)\to M_T(\alpha_T)\colon 
E\mapsto E_T.$ 
We also see that the map $E\mapsto E_T$
is independent of representatives of S-equivalence
class for $\theta_0$-stability up to isomorphism
of $T$-modules.
Hence we get the morphism 
$\pi_0^r\colon M_0(\alpha_r)\to M_T(\alpha_T)$ 
and this map is set theoretically injective. 
\begin{lemm}\label{clo}
The morphism $\pi_0^r\colon M_0(\alpha_r)\to M_T(\alpha_T)$ 
gives a closed embedding.
\begin{proof}
$\pi_0^r$ is induced from a homomorphism of graded rings of
invariant sections, therefore affine morphism. Since both
of $M_0(\alpha_r)$ and $M_T(\alpha_T)$ are projective, $\pi_0^r$ is finite.
Since $\pi_0^r$ is set theoretically injective,
the claim holds.
\end{proof}
\end{lemm}

Furthermore we easily see that morphisms $$\pi_-^{n+1}\colon M_-(\alpha_{n+1})\to M_T(\alpha_T), \pi_+^{n-2}\colon M_+(\alpha_{n-2})\to M_T(\alpha_T)$$ are isomorphisms. 
Inverse maps $$E_T=\left(\C^n\stackrel{A_i}{\to}\C^{n-1}\right)\mapsto E_\pm=\left(\C^n\stackrel{C_i^\pm}{\to}\C^{2n-1+r}\stackrel{D_j^\pm}{\to}\C^{n-1}\right)$$ of $\pi^{n-2}_+$ and $\pi^{n+1}_-$ are defined by
$$(C_0^+,C_1^+,C_2^+)=
\begin{pmatrix}
0&-A_2&A_1\\
A_2&0&-A_0\\
-A_1&A_0&0
\end{pmatrix}
, 
\begin{pmatrix}
D_0^+\\D_1^+\\D_2^+
\end{pmatrix}
=I_{3n-3}
\text{ for }\pi^{n-2}_+$$
and
$$(C_0^-,C_1^-,C_2^-):=I_{3n}, 
\begin{pmatrix}
D^-_0\\D^-_1\\D^-_2
\end{pmatrix}
:=\begin{pmatrix}
0&-A_2&A_1\\
A_2&0&-A_0\\
-A_1&A_0&0
\end{pmatrix}
\text{ for }\pi^{n+1}_-,$$ 
where $I_{3n}$ and $I_{3n-3}$ are unit matrices with sizes $3n$ and $3n-3$ respectively.

Hence we get the diagram:
\begin{equation}\label{kro}
\xymatrix{
M_+(\alpha_r)\ar[dr]^{f_+}\ar[dd]_{g_+}&&
\ar[dd]^{g_-}\ar[dl]_{f_-}M_-(\alpha_r)\\
&M_0(\alpha_r)\ar@{->}[d]^{\pi_0^r}&\\
M_+(\alpha_{n-2})\ar[r]^{\cong}_{\pi_+^{n-2}}&
M_T(\alpha_T)&\ar[l]_{\cong}^{\pi_-^{n+1}}M_-(\alpha_{n+1})}
\end{equation} 
where $g_-:=(\pi_-^{n+1})^{-1}\circ\pi_0^r\circ f_-$
and $g_+:=(\pi_+^{n-2})^{-1}\circ\pi_0^r\circ f_+$.
Morphisms $g_-$ and $g_+$ are explicitly defined by the following universal extensions for each $E_-\in M_-(\alpha_r)$ and $E_+\in M_+(\alpha_r)$, 
\begin{equation}\label{ex2}
0\to \Ext^1_B(E_-,S_0)^{\vee}\otimes S_0\to g_-(E_-)\to E_-\to 0,
\end{equation}
\begin{equation}\label{ex3}
0\to E_+\to g_+(E_+)\to \Ext^1_B(S_0,E_+)\otimes S_0\to 0.
\end{equation}
Hence via isomorphisms (\ref{iso}), the morphism $g_-$ coincides with Yoshioka's map $\psi$, which is defined by a similar exact sequence (\ref{ex1}). 
By Lemma~\ref{lemkro} and the diagram~(\ref{kro}), we have 
$$M_0(\alpha_r)\cong \im g_-\cong \im \psi.$$
This gives a proof of (1) in Proposition~\ref{inpro2}.

\subsection{Brill-Noether locus}
\label{subsec:3-2}
We introduce
the Brill-Noether locus $M_-^i(\alpha_r)$ and 
$M_+^i(\alpha_r)$ as follows.
$$M_-^i(\alpha_r):=\{E_-\in M_-(\alpha_r)\mid
\dim_{\C}\Hom_B(S_0,E_-)=i\},$$
$$M_+^i(\alpha_r):=\{E_+\in M_+(\alpha_r)\mid
\dim_{\C}\Hom_B(E_+,S_0)=i\}.$$
When we replace '=i' by '$\ge$ i'
in the right hand side, the corresponding 
moduli spaces are denoted by the 
left hand side with 'i' replaced by 
'$\ge$ i'. 

If we put 
$\delta^{\ast}:=\sum_ix_i\circ\delta^{\ast}_i
\colon\U_0\to\U_1\otimes\mathbf x$,
then the zero locus of 
$\wedge^{\rk\U_0-i+1}\delta^{\ast}$ defines 
$M_-^{\ge i}(\alpha_r)$ as a closed subscheme 
of $M_-(\alpha_r)$ because 
$\ker \delta^{\ast}_x\cong \Hom_B(S_0,\U_x)$ for 
any $x\in M_-(\alpha_r)$, where $\U=\left(
\U_{-1}\stackrel{\gamma_i^{\ast}}{\to}
\U_0\stackrel{\delta_j^{\ast}}{\to}
\U_1\right)$ is a 
universal family of $B$-modules on $M_-(\alpha_r)$. 
Similarly, 
$M_+^{\ge i}(\alpha_r)$ is defined as a closed subscheme 
of $M_+(\alpha_r)$. $M_-^i(\alpha_r)=M_-^{\ge i}(\alpha_r)
\setminus M_-^{\ge i+1}(\alpha_r)$ and  
$M_+^i(\alpha_r)=M_+^{\ge i}(\alpha_r)
\setminus M_+^{\ge i+1}(\alpha_r)$ are open subset 
of $M_-^{\ge i}(\alpha_r)$ and $M_+^{\ge i}(\alpha_r)$,
respectively.

\subsection{Set-theoretical description of
Grassmann bundles}
\label{subsec:3-3}
By Lemma~\ref{lemkro} we have the following 
proposition.

\begin{prop}\label{set}
The following hold.
\begin{itemize}
\item[(1)] For any $E_-\in M_-^i(\alpha_r)$, we put $E':=\coker\left(\Hom_B(S_0,E_-)\otimes S_0\to E_-\right)$. 
Then $E'$ is $\theta_-$-semistable and $\Hom_B(S_0,E')=0$, that is, $E'\in M_-^0(\alpha_{r-i})$. Hence $E'$ is also $\theta_0$-stable.
\item[(2)] Conversely, for any $E'\in M_-^0(\alpha_{r-i})$ and any $i$-dimensional vector subspace $V\subset\Ext^1_B(E',S_0)$, we obtain a $B$-module $E_-$ by the canonical exact sequence 
$$0\to V^{\vee}\otimes S_0\to E_-\to E'\to 0.$$
Then $E_-$ is $\theta_-$-semistable and $\Hom_B(S_0,E_-)\cong V$, that is, $E_-\in M_-^i(\alpha_r)$. 
\item[(3)] For any $E_+\in M_+^i(\alpha_r)$, we put $E':=\ker\left(E_+\to \Hom_B(E,S_0)^{\vee}\otimes S_0\right)$. 
Then $E'$ is $\theta_+$-semistable and $\Hom_B(E',S_0)=0$, that is, $E'\in M_+^0(\alpha_{r-i})$. Hence $E'$ is also $\theta_0$-stable.
\item[(4)] Conversely, for any $E'\in M_+^0(\alpha_{r-i})$ and any $i$-dimensional vector subspace $V\subset\Ext^1_B(S_0,E')$, we obtain a $B$-module $E_+$ by the canonical exact sequence 
$$0\to E'\to E_+\to V\otimes S_0\to 0.$$
Then $E_+$ is $\theta_+$-semistable and $\Hom_B(E_+,S_0)\cong V$, that is, $E_+\in M_+^i(\alpha_r)$. 
\end{itemize}
\end{prop}
By Lemma~\ref{lemkro}, $M_-^0(\alpha_{r-i})$ is 
set theoretically equal to $M_+^0(\alpha_{r-i})$. 
For any $B$-module $E$, we have $\Ext^2_B(S_0,E)=
\Ext^2_B(E,S_0)=0$ by \cite[Lemma~4.6~(1)]{O}.
Hence by the Rieman-Roch formula,
for any element 
$E'\in M_-^0(\alpha_{r-i})
=M_+^0(\alpha_{r-i})$ we have
$\dim_{\C}\Ext^1_B(E',S_0)=n+1-r+i$ 
and $\dim_{\C}\Ext^1_B(S_0,E')=
n-2-r+i$. 
If $n-2-r\ge 0$, then by the above lemma 
we have set theoretical equalities 
\begin{equation}\label{seteq}
\begin{split}
f_-\left(M_-^i(\alpha_r)\right)
&=\left\{S_0^{\oplus i}\oplus E'\mid 
E'\in M_-^0(\alpha_{r-i})=M_+^0(\alpha_{r-i})\right\}/\equiv_S\\
&=f_+\left(M_+^i(\alpha_r)\right),
\end{split}
\end{equation}
where $\equiv_S$ denotes the S-equivalence relation (cf. \cite[\S~4.1]{O}).
This gives a proof of (1) of Main~Theorem~\ref{inmain}.
Fibers of S-equivalence class of 
$S_0^{\oplus i}\oplus E'$ by $f_-$ and $f_+$ are parametrized 
by $Gr(\Ext^1_B(E',S_0),i)$ and $Gr(\Ext^1_B(S_0,E'),i)$
for $E'\in M_-^0(\alpha_{r-i})=M_+^0(\alpha_{r-i})$. 

\begin{lemm}For any 
integer $i>r$ the following holds.
$$M_-^i(\alpha_r)=M_+^i(\alpha_r)=\emptyset$$
\end{lemm}
\begin{proof}
By \cite[Lemma~5.7]{Y2}, we have 
$M_-^i(\alpha_r)=\emptyset$ for any
$i>r$.
By (\ref{seteq}) this implies
$M_+^i(\alpha_r)=\emptyset$.
\end{proof}

\subsection{Description of $C_{\PP^2}$}
\label{subsec:3-4}
In the following proposition we use the symbol 
$\alpha_{r,n}={}^t(n,2n-1+r,n-1)\in K(B)$.
\begin{prop}\label{cha}
The following hold.
\begin{itemize}
\item[(1)] $\MPrn\neq\emptyset$ if and only if
$n\ge r-1$.
\end{itemize}
In the following, we assume 
$r\ge 2$.
\begin{itemize}
\item[(2)] $W_0=\R_{\ge 0}\theta_0$ is a wall on
$\alpha_{r,n}^{\perp}\otimes\R$ for $n\ge r-1$.
\item[(3)] $W_{\PP^2}=\R_{\ge 0}\theta_{\PP^2}$ is a wall on
$\alpha_{r,n}^{\perp}\otimes\R$ for $n\ge r$.
\end{itemize}
Hence we have $\R_{> 0}\theta_0+\R_{> 0}\theta_{\PP^2}
=C_{\PP^2}$ if $n\ge r$.
\end{prop}
\begin{proof}
(1) By the criterion for the existence of 
non exceptional stable sheaves in \cite[\S~16.4]{P1},
we have our claim.\\
(2) We assume $n\ge r-1$.
By (1), there exists an element $E$ of 
$M_-(\alpha_{r-1,n})\cong\MP(r-1,1,n)$. 
By Lemma~\ref{lemkro}~(1), 
a $B$-module $E\oplus S_0$
is $\theta_0$-semistable and has a submodule
$S_0$ with $\theta_0(S_0)=0$.
Hence $W_0=\R_{\ge 0}\theta_0$ is a wall on
$\alpha_{r,n}^{\perp}\otimes\R$.\\
(3) We assume $n\ge r$ and 
take an element $\F$ of $\MP(r,1,n-1)$.
We consider the exact sequence
$$0\to\F'\to\F\to \mo_x\to0$$
for skyscraper sheaf $\mo_x$ at any
point $x\in\PP^2$.
Then $\F'$ is stable since $\mu(\F')=\mu(\F)$.
This gives elements $F:=\Phi(\F[1])$,
$F':=\Phi(\F'[1])$ of $M_-(\alpha_{r,n-1})$, 
$M_-(\alpha_{r,n})$ respectively and an exact sequence
of $B$-modules 
$$0\to\Phi(\mo_x)\to F'\to F\to0.$$
Hence $W_{\PP^2}=\R_{\ge0}\theta_{\PP^2}$ is
a wall on $\alpha_r^{\perp}\otimes\R$.
These together with (\ref{chaeq}) imply the last assertion.
\end{proof}
By this proposition and \cite[Lemma~6.3~(2)]{O} we have $\R_{> 0}\theta_0+\R_{> 0}\theta_{\PP^2}
=C_{\PP^2}$ if $r=1$ and $n\ge 2$, or $r\ge 2$ and $n\ge r$. 
In this case $C_+$ is different from $C_-=C_{\PP^2}$ and it is adjacent to $C_-=C_{\PP^2}$ with the boundary containing $W_0$.
By the description of the canonical bundle of $M_-(\alpha_r)$ in Proposition~\ref{pic} we see that the diagram (\ref{thad}) gives the flip of $M_-(\alpha_r)$. 
Hence we get an isomorphism $M_+(\alpha_r)\cong M_+(r,1,n)$ and a proof of (2) in Proposition~\ref{inpro2}.
\begin{prop}\label{sm}
Moduli schemes $M_-(\alpha_r)$ and $M_+(\alpha_r)$ are smooth.
\end{prop}
\begin{proof}
By Lemma~\ref{def}, deformation functors of $B$-modules $E$ have obstruction theories with values in $\Ext^2_{B}(E,E)$. 
Since $\Ext^2_{B}(E_-,E_-)=0$ for any $E_-\in M_-(\alpha_r)\cong M_{\PP^2}(r,1,n)$ we see that $M_-(\alpha_r)$ is smooth (cf. \cite[Corollary~4.5.2]{HL}). 
Furthermore if $E_+$ is an element of $M_+(\alpha_r)$ then by Proposition~\ref{set} we have an exact sequence 
$$0\to E'\to E_+\to\C^i\otimes S_0\to 0$$ for some $i$ and $E'\in M_-^0(\alpha_{r-i})$.
Since $$\Ext^2_{B}(S_0,E_+)=\Ext^2_{B}(E',E')=\Ext^2_{B}(E',S_0)=0,$$ we also have $\Ext^2_B(E_+,E_+)=0$. 
Thus $M_+(\alpha_r)$ is also smooth.
\end{proof}
In the rest of this section we show that the diagram~(\ref{thad}) is scheme theoretically described by stratified Grassmann bundles.

\subsection{Coherent systems}
\label{subsec:3-5}
For $r\ge i\ge 0$ we define moduli of coherent 
systems $M_-(\alpha_r,i)$ and $M_+(\alpha_r,i)$ :
$$M_-(\alpha_r,i):=\{(E_-,V)\mid E_-\in
M_-(\alpha_r), V\subset\Hom_B(S_0,E_-)
\text{ with } \dim_{\C}V=i\},$$
$$M_+(\alpha_r,i):=\{(E_+,V)\mid E_+\in
M_+(\alpha_r), V\subset\Hom_B(E_+,S_0)
\text{ with } \dim_{\C}V=i\}.$$
These moduli schemes are constructed as follows. 
We only show the construction of $M_-(\alpha_r,i)$ 
because the construction of $M_+(\alpha_r,i)$ is 
similar. 

We introduce the following quiver with relations
$(\bar{Q},I)$, where
$$\bar{Q}:=
\begin{minipage}{6cm}
$\xymatrix{
&&\stackrel{u}{\bullet}\ar[d]_{\rho}&&\\
\stackrel{\bar{v}_{-1}}{\bullet}&&\ar[ll]_{\bar{\gamma}_i}
\ar[d]_{\iota}\stackrel{\bar{v}_0}{\bullet}&&\ar[ll]_{\bar{\delta}_j}
\stackrel{\bar{v}_1}{\bullet},\\
&&\stackrel{w}{\bullet}&&}$
\end{minipage}
 (i,j=0,1,2)$$
and $I$ is generated by the following relations
$$\bar{\gamma}_i\rho=\iota\bar{\delta}_j=
\bar{\gamma}_i\delta_j+\bar{\gamma}_j\delta_i
=\iota\rho=0, (i,j=0,1,2).$$
Let $\bar{B}$ be a path algebra $\C\bar{Q}/I$ 
of the quiver with relations $(\bar{Q},I)$. 
We have simple modules $\C \bar{v}_{-1}$,
$\C \bar{v}_{0}$,
$\C \bar{v}_{1}$, $\C u$ and $\C w$. 
For each $\alpha_r\in K(B)$,  
we put 
$$\bar{\alpha}_r:=n[\C\bar{v}_{-1}]+(2n+r-1)[\C
\bar{v}_0]+(n-1)[\C\bar{v}_1]+(2n+r-i)[\C u]
+i[\C w]\in K(\bar{B}),$$
and for $\theta_-=(\theta_-^{-1},
\theta_-^{0},\theta_-^1)\in \alpha_r^{\perp}$
and $\e'>0$ small enough,
we put 
$$\bar{\theta}_-:=\theta_-^{-1}[\C\bar{v}_{-1}]^{\ast}+
\theta_-^0[\C\bar{v}_0]^{\ast}+\theta_-^1[\C\bar{v}_1]^{\ast}
+\frac{\e'}{2n+r-i-1}[\C u]^{\ast}-\frac{\e'}{i}[\C w]^{\ast}
\in\bar{\alpha}_r^{\perp}.$$

For any right $\bar{B}$-module $\bar{E}_-$ with 
$[\bar{E}_-]=\bar{\alpha}_r\in K(\bar{B})$, 
$$\bar{E}_-=
\begin{minipage}{6cm}
$\xymatrix{
0\ar[rr]\ar@{}[drr]|\circlearrowleft&&
\C^{2n-1+r-i}\ar[rr]
\ar@{}[drr]|\circlearrowleft&&0\\
\ar@{}[drr]|\circlearrowleft
\ar[u]\C^n\ar[rr]^{\bar{C}_i}&&
\ar[u]^B\C^{2n-1+r}\ar[rr]^{\bar{D}_j}
\ar@{}[drr]|\circlearrowleft&&\ar[u]\C^{n-1}\\
\ar[u]0\ar[rr]&&
\ar[u]^{A}\C^i\ar[rr]
&&\ar[u]0}$
\end{minipage}
$$
we put 
$$E_-:=
\begin{minipage}{6cm}
$\left(\xymatrix{
\C^n\ar[rr]^{\bar{C}_i}&&
\C^{2n-1+r}\ar[rr]^{\bar{D}_j}
&&\C^{n-1}}\right).$
\end{minipage}$$ 
The following lemma is proved similarly
as in Lemma~\ref{lemkro}~(3).
\begin{lemm}
If we take $\e'$ small enough, then
$\bar{E}_-$ is $\bar{\theta}_-$-semistable 
if and only if $E_-$ is $\theta_-$-semistable
and $A$ is injective 
and $B$ is surjective.
\end{lemm}
Hence if we denote by $M_{\bar{B}}(\bar{\alpha}_r,\bar{\theta}_-)$
the moduli of $\bar{\theta}_-$-semistable 
$\bar{B}$-module $\bar{E}_-$ with $[\bar{E}_-]=
\bar{\alpha}_r$, 
we get an isomorphism $M_{\bar{B}}(\bar{\alpha}_r,\bar{\theta}_-)
\cong M_-(\alpha_r,i)$. We write as $\bar{E}_-=(E_-,\C^i)
\in M_-(\alpha_r,i)$ by abbreviation.

We have morphisms 
$$q_1\colon M_-(\alpha_r,i)\to M_-(\alpha_r)\colon
\bar{E}_-=(E_-,\C^i)\mapsto E_-$$ and 
$$q_2\colon M_-(\alpha_r,i)\to
M_-(\alpha_{r-i})\colon \bar{E}_-\mapsto
q_2(\bar{E}_-)$$ defined by the canonical exact sequence
$$0\to\C^i\otimes S_0\to E_-\to q_2(\bar{E}_-)\to 0.$$

Similarly we have morphisms $q_1'\colon M_+(\alpha_r,i)
\to M_+(\alpha_r)$
and $q_2'\colon M_+(\alpha_r,i)\to M_+(\alpha_{r-i})$. 
If we take an element $\bar{E}_+:=(E_+,\C^i)\in
M_+(\alpha_r,i)$, then
$q_1'$ and $q_2'$ are defined by $q_1'(\bar{E}_+)=E_+$ 
and $q_2'(\bar{E}_+):=\ker\left(E_+\to
(\C^i)^{\ast}\otimes S_0\right)$. 
\begin{prop}\label{grbun}
The following hold.
\begin{itemize}
\item[(1)] The morphism $q_1\colon M_-(\alpha_r,i)\to M_-(\alpha_r)$ is a $Gr(j,i)$-bundle over each strata $M^j_-(\alpha_r)$. 
In particular we have an isomorphism $$q_1\colon q_1^{-1}(M^i_-(\alpha_r))\cong M^i_-(\alpha_r).$$
\item[(2)] The morphism $q_2\colon M_-(\alpha_r,i)\to M_-(\alpha_{r-i})$ is a $Gr(n+1-r+i,i)$-bundle. In particular, we have an isomorphism $q_2\colon M_-(\alpha_{n+1},i)\cong  M_-(\alpha_{n+1-i})$.
\item[(3)] For any $j\ge 0$, we have 
$q_1^{-1}(M_-^{i+j}(\alpha_r))\cong
q_2^{-1}(M_-^{j}(\alpha_{r-i}))$.
\item[(4)] The morphism $q'_1\colon M_+(\alpha_r,i)\to M_+(\alpha_r)$ is a $Gr(j,i)$-bundle over each strata $M_+^j(\alpha_r)$. 
In particular we have an isomorphism $$(q'_1)^{-1}(M_+^j(\alpha_r))\cong M_+^j(\alpha_r).$$
\item[(5)] The morphism $q'_2\colon M_+(\alpha_r,i)\to M_+(\alpha_{r-i})$ is a $Gr(n-2-r+i,i)$-bundle. In particular we have an isomorphism $q'_2\colon M_+(\alpha_{n-2},i)\cong M_+(\alpha_{n-2-i})$.
\item[(6)] For any $j\ge 0$, we have 
${q_1'}^{-1}(M_+^{i+j}(\alpha_r))\cong
{q_2'}^{-1}(M_+^{j}(\alpha_{r-i}))$.
\end{itemize}
\end{prop}
\begin{proof}
(1)  The fiber of $q_1$ over $E_-\in M^j_-(\alpha_r)$ is parametrized by $Gr(\Hom_B(S_0,E_-),i)$ for all $j\ge i$. 
For the universal bundle $\U$ in (\ref{univ1}), as in \S~\ref{subsec:3-3} we put $\delta^{\ast}:=\sum_i\delta^{\ast}_i \otimes x_i\colon\U_0\to\U_1\otimes\mathbf x$.
Then for any point $p\in M_-(\alpha_r)$, we have $\Hom_B(S_0,\U_p)\cong (\ker \delta^{\ast})_p$.
Since $\ker \delta^{\ast}$ is locally free of rank $j$ on $M^j_-(\alpha_r)$ (cf. \cite[Chapter II]{ACGH}), we have $Gr(j,i)$-bundle $Gr(\ker \delta^{\ast}|_{M_-^j(\alpha_r)},i)$ on $M^j_-(\alpha_r)$.

On the other hand, by the definition of 
$M^j_-(\alpha_r)$ (cf \S~\ref{subsec:3-2}),
we easily see that $M^j_-(\alpha_r)$
represents the moduli functor 
parametrizing families of $\theta_-$-semistable
$B$-modules $E_-$ with $[E_-]=\alpha_r$ and 
$\dim_{\C}\Hom_B(S_0,E_-)=j$. 
Hence $q_1^{-1}(M^j_-(\alpha_r))$ have the same 
universal property of 
$Gr(\ker \gamma^{\ast}|_{M_-^j(\alpha_r)},i)$ and
we have $q_1^{-1}(M^j_-(\alpha_r))
\cong Gr(\ker \gamma^{\ast}|_{M^j_-(\alpha_r)},i)$.\\
(2) The fiber of $q_2$ over $E'=q_2(\bar{E}_-)$ 
is parametrized by $Gr\left(\Ext^1_B(E',S_0),i\right)$.
For the universal family 
$\U'=\left(\U'_{-1}\stackrel{{\gamma'}^{\ast}_i}
{\to}\U'_0\stackrel{{\delta'}^{\ast}_j}{\to}\U'_1\right)$
of $B$-modules on $M_-(\alpha_{r-i})$,
we put $${\gamma'}^{\ast}:=\sum_i
{\gamma'}^{\ast}_i\otimes x^{\ast}_i\colon
\U_{-1}\otimes\mathbf x\to\U_0.$$ 
Since we have $(\ker{\gamma'}^{\ast})^{\vee}_{p'}
\cong \Ext^1_B(\U'_{p'},S_0)$ for any 
$p'\in M_-(\alpha_{r-i})$.
Similarly as in (1) we get $$M_-(\alpha_r,i)\cong 
Gr\left((\ker{\gamma'}^{\ast})^{\vee},i\right).$$\\
(3) Since spaces of both sides have the same 
universal property, our claim holds.\\
(4), (5) and (6) are proved similarly as in 
(1), (2) and (3).
\end{proof}

\begin{coro}
$M^i_-(\alpha_r)$ and $M^i_+(\alpha_r)$
are smooth for any $i,r\ge 0$.
\end{coro}
\begin{proof}
The restriction of the morphism $q_1\colon M_-(\alpha_r,i)\to M_-(\alpha_r)$ gives an isomorphism $$q_1^{-1}(M_-^i(\alpha_r))\cong M_-^i(\alpha_r).$$
By Proposition~\ref{grbun}~(3) we have an isomorphism $M_-^i(\alpha_r)\cong q_2^{-1}(M_-^0(\alpha_{r-i}))$.
Hence by Proposition~\ref{grbun}~(2), $M_-^i(\alpha_r)$ is isomorphic to a Grassmann-bundle over $M_-^0(\alpha_{r-i})$.
Since $M_-^0(\alpha_{r-i})$ is smooth by Proposition~\ref{sm}, we see that $M_-^i(\alpha_r)$ is also smooth.
Similarly $M_+^i(\alpha_r)$ is shown to be smooth.
\end{proof}

\subsection{Stratified Grassmann bundle}
\label{subsec:3-6}
In this section we show that morphisms $f_{\pm}\colon M_{\pm}(\alpha_r)\to M_0(\alpha_r)$ are described by stratified Grassmann bundles using Proposition~\ref{grbun}.

We consider the diagram: $$\xymatrix{&\ar[dl]_{\cong}^{q_2}M_-(\alpha_{n+1},n+1-r)\ar[dr]_{q_1}&\\M_-(\alpha_{r})&&M_-(\alpha_{n+1}).
}$$
By Proposition~\ref{grbun}~(2), $q_2$ is an isomorphism and we have a map $q_1\circ q_2^{-1}\colon M_-(\alpha_r)\to M_-(\alpha_{n+1})$,
which coincides with $g_-$ by (\ref{ex2}).
This gives another proof of Theorem~\ref{yos}. 
Similarly the map $q_1'\circ{q_2'}^{-1}\colon M_+(\alpha_r)\cong M_+(\alpha_{n-2},n-2-r)\to M_+(\alpha_{n-2})$ coincide with the map $g_+$ by (\ref{ex3}).  
For any $r\ge 0$, we have isomorphisms $M^0_-(\alpha_r)\cong M^{n+1-r}_-(\alpha_{n+1})$ and $M^0_+(\alpha_r)\cong M^{n-2-r}_+(\alpha_{n-2})$ via $g_-$ and $g_+$ respectively.
In particular, we have isomorphisms $$M^0_-(\alpha_{r-i})\cong M^{n+1-r+i}_-(\alpha_{n+1}), M^0_+(\alpha_{r-i})\cong M^{n-2-r+i}_+(\alpha_{n-2}).$$
By the diagram~(\ref{kro}), $M^{n+1-r+i}_-(\alpha_{n+1})$ and $M^{n-2-r+i}_-(\alpha_{n-2})$ coincide with images $f_-(M_-^i(\alpha_r))\cong f_+(M_+^i(\alpha_r))$ of $f_-$ and $f_+$. 
This gives a proof of (2) in Main Theorem~\ref{inmain}.

By Proposition~\ref{grbun} and the diagram~(\ref{kro}) we also have proofs of (3) and (4) in Main Theorem~\ref{inmain}.

\subsection{Hodge polynomials of flips}
\label{subsec:3-7}
We study the difference between Hodge polynomials of $M_-(\alpha_r)$ and $M_+(\alpha_r)$. 
To do this we use the virtual Hodge polynomial $e(Y):=\sum_{p,q}e^{p,q}(Y)x^py^q$ for any variety $Y$ (cf. \cite{DK}).

By Main Thoerem~\ref{inmain} we get the following diagram
\begin{equation}
\xymatrix{
\sqcup M_+^i(\alpha_r)\ar[rd]^{f_+}& &\sqcup M_-^i(\alpha_r) \ar[ld]_{f_-}\\
&\sqcup M_0^i(\alpha_r)&}
\end{equation}
where restrictions of $f_+$ and $f_-$ to $M_\pm^i$ are $Gr(n-2-r+i,i)$-bundle and $Gr(n+1-r+i,i)$-bundle over $M^i_0(\alpha_r)\cong M_\pm^0(\alpha_{r-i})$, respectively.
Hence we get the following equality.
\begin{multline}\label{diff}
e\left(M_-(\alpha_r)\right)-e\left(M_+(\alpha_r)\right)=\\
\sum_{i>0}\Big(e\left(Gr(n+1-r+i,i)\right)-e\left(Gr(n-2-r+i,i)\right)\Big)e\left(M_-^0(\alpha_{r-i})\right).
\end{multline}

In the following we compute the Hodge polynomial of $M_+(\alpha_r)$ from that of $M_-(\alpha_r)$ in the case where $r=1,2$. 
In this case, we know the Hodge polynomial of $M_-(\alpha_r)\cong M_{\PP^2}(r,1,n)$ from \cite{ES} and \cite{Y1}.
We need the following proposition. 
\begin{prop}\label{p2}
We have following isomorphisms:
$$M_-(\alpha_0)\cong M_+(\alpha_0)\cong\PP^2.$$
\end{prop}
A proof of this proposition is given in Appendix.
From this proposition and (\ref{diff}), we get the following: $$\cdots=M_+(1,1,1)=M_+(1,1,2)=\emptyset,$$
\begin{equation*}
\begin{split}
e\left(M_+(1,1,3)\right)&=t^{12}+t^{10}+3t^8+3t^6+3t^4+t^2+1,\\
e\left(M_+(1,1,4)\right)&=t^{16}+2t^{14}+5t^{12}+8t^{10}+10t^8+8t^6+5t^4+2t^2+1,\\
e\left(M_+(1,1,5)\right)&=\cdots+21t^{10}+19t^8+11t^6+6t^4+2t^2+1,\\
e\left(M_+(1,1,n)\right)&=e\left(M_{\PP^2}(1,1,n)\right)-
(t^{2n+4}+2t^{2n+2}+3t^{2n}+2t^{2n-2}+t^{2n-4}),
\end{split}
\end{equation*}
and $$\cdots=M_+(2,1,1)=M_+(2,1,2)=M_+(2,1,3)=\emptyset,$$
\begin{equation*}
\begin{split}
e\left(M_+(2,1,4)\right)=\cdots&+12t^{12}+10t^{10}+8t^8+5t^6+3t^4+t^2+1,\\
e\left(M_+(2,1,5)\right)=\cdots&+67t^{16}+60t^{14}+48t^{12}+32t^{10}+20t^8+10t^6\\
&+5t^4+2t^2+1,
\end{split}
\end{equation*}
where $t=xy$.

\appendix

\section{Proof of Proposition~\ref{p2}}

We take $\alpha_0\in K(\PP^2)$ such that 
$\ch(\alpha_0)=-(0,1,n)$ as in \S3 and give a
proof of Proposition~\ref{p2}.

\subsection{Bridgeland stability}
We briefly introduce 
the concept of Bridgeland stability. 
For details the reader can consult \cite{B2}.
Let $\A$ be an abelian category, $K(\A)$ the
Grothendieck group of $\A$. 
\begin{defi}
A stability function $Z$ on $\A$ is a 
group homomorphism from $K(\A)$ to $\C$
satisfying that for any object $E\in\A$,
if $E$ is not equal to zero we have 
$Z(E)\in\R_{>0}\exp(\sqrt{-1}\pi\phi(E))$ 
with $0<\phi(E)\le 1$.
\end{defi}
The real number $\phi(E)$ is called phase of 
$E$. 
\begin{defi}
A nonzero object $E\in\A$ is semistable with respect to $Z$
if and only if for any proper subobject $0\neq
F\subsetneq E$ we have $\phi(F)\le\phi(E)$. If the
inequality is always strict we call $E$ to be 
stable with respect to $Z$.
\end{defi}
Let $\T$ be a triangulated category,
$K(\T)$ the Grothendieck group of $\T$.
\begin{defi}
A stability condition $\sigma$ on $\D(\PP)$ 
is a pair $\sigma=(\A,Z)$, which consists of
a full subcategory $\A$ of $\T$ and a group
homomorphism $Z\colon K(\T)\to\C$ satisfying
following conditions:\\
\begin{itemize}
\item $\A$ is a heart of a bounded t-structure of 
$\T$, which implies $\A$ is an abelian category and 
$K(\A)$ is isomorphic to $K(\T)$ by the inclusion
$\A\subset\T$. Hence we always identify them.\\
\item $Z$ is a stability function on $\A$ 
via the above identification $K(\A)=K(\T)$. \\
\item $Z$ has Harder-Narasimhan property.
\end{itemize}
\end{defi}
We omit the definition of ``a heart of a bounded 
t-structure'' and ``Harder-Narasimhan property''
(see \cite[\S~2 and \S~3]{B2}). 
We denote a set of all stability conditions satisfying
a technical condition called ``local finiteness'' 
(see \cite[\S~5]{B2})
by $\Stab(\T)$. 

\begin{defi}
For a stability condition $\sigma=(\A,Z)\in\Stab(\T)$,
an object $E\in \T$ is called $\sigma$-(semi)stable
if and only if $E$ belongs to 
$\A$ up to  
shift functors $[n]\colon\T\to\T$ for $n\in\Z$, 
and it is semistable with respect to $Z$. 
\end{defi}

In the following we only consider the case where 
$\T=\D(\PP^2)$ and we put $\Stab(\PP^2):=\Stab(\T)$. 
For $\alpha\in K(\PP^2)$ and 
$\sigma=(\A,Z)\in\Stab(\PP^2)$, we define a moduli functor 
$\M_{\D(\PP^2)}(\alpha,\sigma)$ of $\sigma$-semistable
objects $E\in\A$ with $[E]=\alpha\in K(\PP^2)$ as follows.
The moduli functor $\M_{\D(\PP^2)}(\alpha,\sigma)$
is a functor from $(\text{Sch}/\C)$ to $(\text{Set})$.
For a scheme $S$ over $\C$ it sends $S$ to a set
$\M_{\D(\PP^2)}(\alpha,\sigma)(S)$ of families 
$\F\in \D(\PP^2\times S)$ of $\sigma$-semistable 
objects with class $\alpha$ in $K(\PP^2)$. This means 
that for any $\C$-valued point $s\in S$, the fiber
$\mathbf L \iota^{\ast}_s\F\in D^-(\PP^2)$ belongs
to the full subcategory $\A\subset \D(\PP^2)$ and
$\sigma$-semistable with $[\mathbf L\iota^{\ast}_s
\F]=\alpha\in K(\PP^2)$. 

There exists a right action of $\grp$ on $\Stab(\PP^2)$ and this action does not change semistable objects.
Hence for any $\alpha\in K(\PP^2)$, $\sigma\in\Stab(\PP^2)$ and $g\in\grp$, there exists an integer $n\in\Z$ such that shift [n] induces an isomorphism of functors 
\begin{equation*}
\M_{\D(\PP^2)}(\alpha,\sigma)\cong \M_{\D(\PP^2)}((-1)^n\alpha,\sigma g) \colon E\mapsto E[n].
\end{equation*}

\subsection{Geometric stability}
Let $H$ be the ample generator of $\Pic(\PP^2)$ and $s,t\in\R$ with $t>0$. 
For any torsion free sheaf $E$ on $\PP^2$, the slope of $E$ is defined by $\mu_H(E):=\frac{c_1(E)}{\rk(E)}$ and define $\mu_H$-semistability. 
$E$ has the Harder-Narasimhan filtration with $\mu_H$-semistable factors.
We denote the maximal value and the minimal value of slopes of $\mu_H$-semistable factors of $E$ by $\mu_{H-\max}(E)$ and $\mu_{H-\min}(E)$, respectively. 
Then we define a pair $\sigma_{(sH,tH)}=(\A_{(sH,tH)},Z_{(sH,tH)})$ as follows. 
\begin{defi}
An object $E\in\D(\PP^2)$ belongs to the full subcategory $\A_{(sH,tH)}$ if and only if 
\begin{itemize}
\item $\mathcal{H}^i(E)=0$ for all $i\neq 0,-1$\\
\item $\mathcal{H}^0(E)$ is torsion or $\mu_{H-\min} (\mathcal{H}^0(E)_{\fr})>st$, where $\mathcal{H}^0(E)_{\fr}$ is the free part of $\mathcal{H}^0(E)$\\
\item $\mathcal{H}^{-1}(E)$ is torsion free and $\mu_{H-\max} (\mathcal{H}^{-1}(E))\le st$.
\end{itemize}
The group homomorphism $Z_{(sH,tH)}$ is defined by $$Z_{(sH,tH)}(E):=-\int_{\PP^2} \ch(E)\exp(-sH-\sqrt{-1}tH).$$
\end{defi}

If $s$ and $t$ belong to $\Q$, then $\sigma_{(sH,tH)}$ is a stability condition on $\D(\PP^2)$ (cf. \cite{ABL}). 
In general we do not know wheather $\sigma_{(sH,tH)}$ is a stability condition on $\D(\PP^2)$. 
We have the following criterion due to Bridgeland.

\begin{prop}{\emph{\bf cf. \cite[Proposition~3.6]{O}}}\label{geom}
For $\sigma\in\Stab(\PP^2)$, there exist $g\in\grp$ and $s,t\in\R$
with $t>0$ such that $\sigma=\sigma_{(sH,tH)}g$ if and only if
the following conditions $($\emph{i}$)$ and $($\emph{ii}$)$ hold.
\begin{itemize}
\item[(i)] For any closed point $x\in\PP^2$, the skyscraper sheaf $\mo_x$
is $\sigma$-stable.
\item[(ii)] For any $\beta\in K(\PP^2)$, if $Z(\beta)=0$ then
$c_1^2-2r\ch_2<0$ where $\ch(\beta)=(r, c_1,\ch_2)$. 
\end{itemize}
\end{prop}

\subsection{Proof of Proposition~\ref{p2}}
We take $\sigma^s=(\A,Z^s)\in\Stab(\PP^2)$ for $s\in\R$ with $-1<s<1$, where $$\A=\langle\mo_{\PP^2}(-1)[2],\mo_{\PP^2}[1], \mo_{\PP^2}\rangle$$ and $Z^s$ is a group homomorphism $Z^s\colon K(\PP^2)\to\C$ defined by $$Z^s(e_{-1})=\frac{-s-1}{2}, Z^s(e_0)=1+\sqrt{-1}, Z^s(e_{1})=\frac{-s+1}{2}$$ for $e_i=[\mo_{\PP^2}(i)[1-i]]\in K(\PP^2)$, $i=-1,0,1$. 
Then by Proposition~\ref{geom} we see that there exists an element $g^s\in\grp$ such that \begin{equation}\label{grp}
\sigma^s=\sigma_{(sH,tH)}g^s,
\end{equation} 
where $t=\sqrt{1-s^2}$.

We take $\alpha_0={}^t(n,2n-1,n-1)\in K(B)$ and define a group homomorphism $\Tilde{\theta}^s \colon K(B)\to \C$ by $$\Tilde{\theta}^s(\beta)=\det\begin{pmatrix}\Re Z^s(\beta)& \Re Z^s(\alpha_0)\\
\Im Z^s(\beta)& \Im Z^s(\alpha_0)\end{pmatrix}$$ for each $\beta\in K(B)$. 
Then by \cite[Proposition~1.2]{O}, $M_B(\alpha_0,\Tilde{\theta}^s)$ corepresents the moduli functor $\mathcal{M}_{\D(\PP^2)}(\alpha_0,\sigma^s)$.
Furthermore by (\ref{grp}), we have an isomorphism 
\begin{equation}\label{modu}
\mathcal{M}_{\D(\PP^2)}(-\alpha_0,\sigma_{(sH,tH)})\cong \mathcal{M}_{\D(\PP^2)}(\alpha_0,\sigma^s)\colon E\mapsto E[1]
\end{equation}
of moduli functors. 
We notice that for an object $E\in\A_{(sH,tH)}$, we have $[E]=-\alpha_0\in K(B)\cong K(\PP^2)$ if and only if $\ch(E)=(0,1,\frac{1}{2}-n)\in \Z\oplus\Z\oplus \frac{1}{2}\Z$.
We give the following lemmas to prove Proposition~\ref{p2} 
\begin{lemm}\label{lemp2}We assume $-1< s\le 0$ and put $t=\sqrt{1-s^2}$.
Then for any line $L\subset\PP^2$, the structure sheaf
$\mo_L(1-n)$ tensored by $\mo_{\PP^2}((1-n)H)$ is $\sigma_{(sH,tH)}$-stable.
\end{lemm}
\begin{proof}
We show that $\mo_L(1-n)\in\A_{(sH,tH)}$ is
$\sigma_{(sH,tH)}$-stable for any line $L\subset\PP^2$.
We take an exact sequence in $\A_{(sH,tH)}$
$$0\to F\to \mo_L(1-n)\to G\to 0.$$
Then we have a long exact sequence 
$$0\to \mathcal{H}^{-1}(G)\to F\to \mo_L(1-n)\to 
\mathcal{H}^0(G)\to 0.$$  
If the dimension of support of $\mathcal{H}^0(G)$
is equal to $1$, we have $\rk(F)=\rk(\mathcal{H}^{-1}(G))$,
$c_1(F)=c_1(\mathcal{H}^{-1}(G))$. If $\rk(F)\neq 0$, this contradicts the
fact that $F,G\in\A_{(sH,tH)}$ implies inequalities
$\mu_{tH}(\mathcal{H}^{-1}(G))\le st<\mu_{tH}(F)$. 
Hence $F$ is a torsion sheaf and $\mathcal{H}^{-1}(G)=0$.
This implies $F=0$ and $G=\mo_L(1-n)$ since $\mo_L(1-n)$
is a pure sheaf.

If the dimension of support of $\mathcal{H}^0(G)$
is equal to $0$, we have $\rk(F)=\rk(\mathcal{H}^{-1}(G))$,
$c_1(F)=c_1(\mathcal{H}^{-1}(G))-1$. 
Inequalities $\mu_H(\mathcal{H}^{-1}(G))\le st<\mu_H(F)$ implis 
$$c_1(F)=1, \text{ and }c_1(\mathcal{H}^{-1}(G))=0.$$ 
Hence we have
$\Im Z_{(sH,tH)}(G)=\rk(\mathcal{H}^{-1}(G))st.$ 
In the case where $s<0$ this implies $\mathcal{H}^{-1}(G)=0$ since
$\Im Z_{(sH,tH)}(G)\ge 0$. In any case, we have $\Im Z_{(sH,tH)}(G)=0$ and	 
$$\phi(G)=1>\phi(\mo_L(1-n)).$$ Hence $G$ does not 
break $\sigma_{(sH,tH)}$-stability of $\mo_L(1-n)$.
\end{proof}

\begin{lemm}
An object $E\in\A_{(0,H)}$ with $[E]=-\alpha_0$ 
is $\sigma_{(0,H)}$-semistable
if and only if $E\cong \mo_{L}(1-n)$ for a line $L$ on $\PP^2$. 
\end{lemm}
\begin{proof}
We assume that $E\in\A_{(0,H)}$ is $\sigma_{(0,H)}$-semistable and 
$\mathcal{H}^{-1}(E)\neq 0$. We put $F:=\mathcal{H}^{-1}(E)[1]$ 
and $G:=\mathcal{H}^0(E)$. Then an exact sequence in $\A_{(0,H)}$
$$0\to F\to E\to G\to0$$ 
implies that $0\le\Im Z_{(0,H)}(F)<\Im Z_{(0,H)}(E)=t\le 1$.
If we put $\ch(F)=-(r,c_1,\ch_2)$ with $r>0$.
Then $\Im Z_{(0,H)}(F)=-c_1=0$ hence 
we see that $\phi(E)<\phi(F)=1$ contradicting to  
$\sigma_{(0,H)}$-semistability of $E$. 
Thus $E$ is a sheaf with $\ch(E)=(0,1,\frac{1}{2}-n)$.
Any subsheaf with support dimension $1$ of $E$ break $\sigma_{(0,H)}$-semistability
of $E$. Hence we see that $E$ is a pure sheaf.
This shows that $E\cong\mo_L(1-n)$ for a line $L$ on $\PP^2$.

Conversely by Lemma~\ref{lemp2},
we see that $\mo_L(1-n)\in\A_{(0,H)}$ is
$\sigma_{(0,H)}$-semistable for any line $L\subset\PP^2$.
\end{proof}

By this lemma and the isomorphism (\ref{modu}), the moduli functor $\M_{\D(\PP^2)}(\alpha_0,\sigma^0)$ is represented by $\PP^2\cong \{\mo_L(1-n)\mid L\subset\PP^2\colon \text{ line }\}$.
By \cite[Proposition~4.4]{O}, $\M_{\D(\PP^2)}(\alpha_0,\sigma^0)$ is also represented by $M_B(\alpha_0,\Tilde{\theta}^0)$.
Hence we have an isomorphism $M_B(\alpha_0,\Tilde{\theta}^0)\cong \PP^2$. 
If we put $s_0:=-\frac{1}{2n-1}$, then we have $\Tilde{\theta}^{s_0+\e}\in C_-$, $\Tilde{\theta}^{s_0-\e}\in C_+$ and  $\Tilde{\theta}^{s_0}\in W_0$ for $\e>0$ small enough.
By Lemma~\ref{lemp2} every object in $M_B(\alpha_0,\Tilde{\theta}^0)$ is $\Tilde{\theta}^s$-stable for $-1<s\le 0$. 
Hence $W_0$ is not a wall and $C_\pm$ and $W_0$ are contained in a single chamber. 
As a consequence we have isomorphisms
$$M_+(\alpha_0)\cong M_-(\alpha_0)\cong M_B(\alpha_0,\Tilde{\theta}^0)\cong\PP^2.$$ 
This completes the proof of Proposition~\ref{p2}.\\

\noindent\emph{
\bf{Acknowledgement}}\\
The author is grateful to his adviser Takao Fujita for many valuable comments and encouragement.
He thanks Hiraku Nakajima and K\={o}ta Yoshioka for valuable comments and motivating him to write this paper.

\bibliographystyle{plain}
\bibliography{b}

\end{document}